\documentclass[Leqno, 12pt]{amsart}

\usepackage[paper=a4paper,left=3.4cm,right=3.4cm,top=3.8cm,bottom=3.8cm]{geometry}
\usepackage{float, graphicx, amssymb, dsfont}
\usepackage{subfig}
\usepackage{tikz}
\usetikzlibrary{arrows,shapes}
\newtheorem{thm}{Theorem}[section]

\newtheorem{lem}[thm]{Lemma}

\newtheorem*{thm*}{Theorem}
\theoremstyle{definition}

\theoremstyle{remark}
\newtheorem{rem}[thm]{Remark}
\newtheorem*{claim}{Claim}
\newtheorem*{ack*}{Acknowledgment}

\numberwithin{equation}{section}
\numberwithin{figure}{section}
\newcommand*{\Perm}[2]{{}_{#1}\!P_{#2}}%
\newcommand{\cA}{\mathcal{A}}       % alphabet
\newcommand{\cC}{\mathcal{C}}
\newcommand{\cD}{\mathcal{D}}

\newcommand{\cI}{\mathcal{I}}

\newcommand{\cN}{\mathcal{N}}

\newcommand{\cP}{\mathcal{P}}
\newcommand{\cQ}{\mathcal{Q}}

\newcommand{\cU}{\mathcal{U}}
\newcommand{\cV}{\mathcal{V}}       % vertex sets
\newcommand{\setC}{\mathbb{C}}
\newcommand{\setN}{\mathbb{N}}

\newcommand{\setZ}{\mathbb{Z}}

\newcommand{\B}{\mathcal{B}}    %block
\newcommand{\onto}{\xymatrix{\ar@{>>}[r]&}}
\newcommand{\da}[4]{\xymatrix{#1 \ar@<.5ex>[r]^{#2} \ar@<-.5ex>[r]_{#3} & #4}}

\newcommand{\floor}[1]{\lfloor#1\rfloor}

\usepackage{mathtools}
\begin{document}
%\nocite{*}

\title[Subshifts with positive entropy dimension]{Constructions of subshifts with positive topological entropy dimension}

\author[U. Jung]{Uijin Jung}
\address{Department of Mathematics \\
	Ajou University \\
	206 Worldcup-ro \\
    Suwon 16499 \\
	South Korea}
\email{uijin@ajou.ac.kr}

\author[J. Lee]{Jungseob Lee}
\email{jslee@ajou.ac.kr}

\author[K. Park]{Kyewon Koh Park}
\address{School of Mathematics \\
	Korea Institute for Advanced Study \\
	85 Hoegiro, Dongdaemun-gu \\
    Seoul 02455 \\
	South Korea}
\email{kkpark@kias.re.kr}

\date{\today}

\maketitle

\begin{abstract}
    %Entropy dimension is introduced to measure the complexity of entropy zero systems.  It measures the subexponential itinery growth rate of orbits.  We construct  an example of positve entropy dimension $\alpha$ for each 0<$\alpha$<1.  They are weakly mixing together with the property of strict ergodicity.  We investgate the Equipartition Property of the examples and show that for a systems of entropy dimension $=alpha$, it satisfies $limsup{1 over n^{\beta}}{(-log \mu(P_n(x)))}$ is equal to 0 if
    %And we also show that the limits do not exist in general. Hence they do not have the Shannon-McMillan-Breiman Theorem. We also investigate the Return Time Property for these examples and show that similar properties hold.  Our examples show that the definition of entropy dimension and slow entropy of using $\epsilon$-balls in the hamming distance are not the same.

    The notion of entropy dimension has been introduced to measure the subexponential complexity of zero entropy systems. In this work we present a general construction of a strictly ergodic subshift of topological entropy dimension $\alpha$ for each $\alpha \in (0,1)$. It is shown that the system satisfies some sort of regularity in the size of atoms and the first return time. Moreover, we modify the construction to obtain a variant system that is weakly mixing.
    %The notion of entropy dimension has been introduced to measure the subexponential complexity of zero entropy systems. In this work we present a general construction of a subshift of positive topological entropy dimension $\alpha$ for each $\alpha \in (0,1)$. We investigate the size of atoms of the system and show that they do not have the equipartition property, however, the size of atoms as well as the return times of the system exhibits some regularity. The systems are strictly ergodic and metrically rigid, and can be modified to be weakly mixing.
\end{abstract}

\section{Introduction}\label{sec:introduction}

    In 1958, Kolmogorov introduced the notion of entropy to dynamical systems generalizing the Shannon entropy in information theory \cite{Sha48}. Entropy measures the chaoticity of a dynamical system. In the case of a $\mathbb{Z}$-action, past is well defined and hence we say entropy measures the (un)predictability of a system knowing the complete past. If it has entropy zero, then we say its past determines the future.  Entropy is an isomorphism invariant and is a complete invariant in the class of Bernoulli actions \cite{Orn70}. Systems of positive entropy have been studied for several decades and many of their properties are well understood at least in the case of $\mathbb{Z}$-actions along with their applications. The entropy theory was extended to general amenable group actions\cite{OrnW83} and recently also to nonamenable group actions \cite{Bow10a, Bow10b}.

    In the study of general group actions, entropy zero systems arise rather naturally. If a general group action has a noncocompact subgroup action of finite entropy, then it is easy to see that its entropy is zero. However their subgroup actions exhibit interesting dynamics with diverse properties. %Given a general group action of zero entropy, like a $\setZ^2$-action, all its subgroup actions may have entropy zero, positive finite entropy or infinite entropy.
    To investigate the properties of entropy zero actions, directional (subgroup) properties have been explored in many different directions \cite{BoyL97, JohS15}. %For example, a zero entropy $\setZ^2$-action may have the following property: (i) all directions have zero entropy, (ii) all directions have positive entropy, or (iii) all of its rational directions have infinite entropy \cite{Par99, Sin85}.
    In general, a zero entropy $\setZ^2$-action may have mixture of positive, infinite and zero directional entropies. We note that if one of the directions has positive entropy, then it has exponential growth rate of orbits.
    Clearly different complexities will give rise to different behaviors of their subgroup actions, and the complexities of general group actions together with their subgroup actions should be explored to understand the dynamics of bigger group actions of entropy zero \cite{RobS11}.

    Entropy zero systems make up a dense $G_{\delta}$ subset of all dynamical systems. However much less is known about the properties of entropy zero dynamical systems. Systems of very small complexity, like group rotations, interval exchange maps, substitution systems and Chacon transformations have been studied and they are known to have polynomial growth rate of orbits.  Some of their properties have been investigated by many authors \cite{Fer97, Goo74, KamZ02}. Recently more examples like Pascal adic transformations \cite{MelP05} and nilpotent group actions \cite{KraHM14} were investigated and shown to have polynomial growth rate.

    Motivated by the study of entropy zero systems of bigger group actions, we want to investigate the properties of entropy zero $\setZ$-actions whose orbit growth rate is subexponential. Only a few systems of entropy zero are known to have intermediate growth rate which is strictly greater than polynomial, but less than exponential  \cite{Cas03}. %To analyze the complexity of these entropy zero systems, we exploit the properties of topological entropy dimension \cite{DouHP13} which was first introduced in \cite{Car97}.
    To analyze the complexity of these entropy zero systems, metric and topological entropy dimension were first introduced in \cite{Car97} and their properties are investigated in \cite{DouHP13} and \cite{Car97}.
    Positive entropy systems have exponential growth rate, hence they have entropy dimension 1, while those of polynomial orbit growth rate have entropy dimension 0. In this article we present a constructive method to build a strictly ergodic subshift of topological entropy dimension $\alpha$ for any $\alpha \in (0,1)$. Roughly speaking, the number of $n$-blocks in the subshift is in the order of $\exp( n^\alpha )$.

    For ergodic systems of positive entropy we have the Shannon-McMillan-Breiman theorem which is called the equipartition property \cite{OrnW83, Shi}.  The Ornstein-Weiss theorem on the return time also holds for those systems \cite{OrnW93}. However for zero entropy systems little is known about these properties except for irrational rotations \cite{KimP08}. Through our example, we investigate the equipartition and return time properties for the systems of subexponential growth rate. We show that it does not have the equipartition property, that is, the size of an atom of $\bigvee_{i=0}^{n-1} T^{-i} \cP$, where $\cP$ is a generator for the system, is not necessarily in the order of $\exp( -{n^\alpha} )$.
    It is not known yet whether there are systems of zero entropy with subexponential growth rate and with corresponding equipartition property.

    %Unlike positive entropy case, smooth realization of entropy zero systems are little known except that any measure preserving diffeomorphism on a circle is conjugate to a rotation. It is an open question whether there exists a smooth system of given entropy dimension $\alpha \in (0,1)$.

    The outline of the article is as follows. Section 2 presents the definitions of the entropy dimension and necessary terminology. In Section 3 we construct strictly ergodic systems with given entropy dimensions. In Section 4 we see that our examples do not have the equipartition property. However we will show that they exhibit some `regularity' in their sizes of atoms and return times. %are able to show that the size of the atoms and the return time have some kind of `regularity'.
    In Section 5 we discuss some other aspects of the systems like rigidity and weakly mixing property.

    %It is not hard to see that all of our system have metric entropy dimension 0. We will discuss a new class of examples together with their metric properties in [JLP].

%\newpage
\vspace{0.4cm}
\section{Background}\label{sec:background}
    We introduce some terminology and known results. For more details on topological entropy dimension, see \cite{DouHP13}.

    Let $(X,T)$ be a topological dynamical system. If $\cU$ and $\cV$ are open covers of $X$, let $\cU \vee \cV = \{ U \cap V : U \in \cU, V \in \cV \}$. Also denote by $\cN(\cU)$ the cardinality of a smallest subcover of $\cU$.

%and $\cU$ be a finite open cover of $X$. Denote by $\cN(\cU)$ the cardinality of a smallest subcover of $\cU$. For each $\alpha > 0$, define
    Given an open cover $\cU$ of $X$ and $\alpha > 0$, define
    \[ \overline D(T, \alpha, \cU) = \limsup_{n \to \infty} \frac{\log \cN( \bigvee_{i=0}^{n-1} T^{-i} \cU )}{n^\alpha}. \]
    Then the function $\overline D(T, \alpha, \cU)$ on the parameter $\alpha$ has a unique critical value, say $\overline D(T,\cU)$, in $[0,\infty]$, i.e., $\overline D(T, \alpha, \cU) = 0 \text{ for } \alpha > \overline D(T,\cU)$ and $\overline D(T, \alpha, \cU) = \infty \text{ for } \alpha < \overline D(T,\cU)$. %The critical value of it is called the \emph{the topological upper entropy dimension of $\cU$}.
    The \emph{(topological) upper entropy dimension} $\overline D(X, T)$ of $(X,T)$ is defined by the supremum of $\overline D(T, \cU)$ for all finite open covers $\cU$ of $X$. It is known that $\overline D(X, T) = \overline D(T, \cU)$ if $\cU$ is a generating open cover. Similarly, the \emph{(topological) lower entropy dimension $\underline D(X,T)$} of $(X,T)$ is defined by using $\liminf$ instead of $\limsup$. If $\overline D(X,T) = \underline D(X,T)$, we denote it by $D(X,T)$ and call it the \emph{(topological) entropy dimension of $(X,T)$}. They are invariant under topological conjugacy.

    The notion of entropy dimension is simplified when a dynamical system is a subshift. A \emph{subshift} (or \emph{shift space}) is a closed $\sigma$-invariant subset of a full shift over some finite set $\cA$ of symbols. For a subshift $X$, denote by $\B_n(X)$ the set of all words of length $n$ appearing in the points of $X$ and $\B(X) = \bigcup_{n \geq 0} \B_n(X)$. %For a finite set $\cA$ and $l \in \setN$, denote by $\cA^l$ the set of all words of length $l$ over $\cA$ and let $\cA^* = \bigcup_{l \geq 0} \cA^l$.
    For a subshift $X$, since there is a natural $0$-th coordinate clopen partition, the topological upper entropy dimension is a unique critical value of the function $\overline D(T, \alpha)$, where
    \[ \overline D(T, \alpha) = \limsup_{n \to \infty} \frac{\log |\B_n(X)|}{n^\alpha}, \]
    and the lower entropy dimension is defined analogously.

    %For more details on topological entropy dimension, see \cite{DouHP13}.

\vspace{0.4cm}

\section{A construction of a subshift with given entropy dimension}\label{sec:construction}

In this section, we present a general construction of a topological dynamical system with positive topological entropy dimension.
The constructed system is a subshift over the alphabet $\{0, 1\}$. We first describe a construction of a subshift with entropy dimension $1/2$ to make the argument readable. A system of topological entropy dimension $\alpha$, for any $0 < \alpha < 1$, can be constructed similarly. In what follows, to simplify the notation we often omit the floor function notation on the square roots and write $\sqrt{N}$ instead of $\floor{\sqrt{N}}$. The omissions will be clear from the context.

It is well known that for all sufficiently large $n$, we have
\begin{equation}\label{stirling}
n \log n - n < \log{ n! } <  n \log n.
\end{equation}
Denote by $\Perm{n}{k}$ the number of $k$-permutations of $n$. We note that
\begin{equation}\label{perm}
k \log n - k < \log{ \Perm{n}{k} } \leq \log n^k = k \log n,
\end{equation}
for all sufficiently large $n$ and any $k$, $1 \leq k \leq n$.
We write $a(n) \sim b(n)$ if the ratio $a(n)/b(n)$ goes to $1$ as $n \to \infty$. From (\ref{stirling}) and (\ref{perm}), we respectively obtain $\log n! \sim n \log n$ and $\log \Perm{n}{k} \sim k \log n$ as $n \to \infty$.

Fix a large square number $l_1 \in \setN$. Let $\cC_1$ be a set of binary words of length $l_1$ with cardinality $N_1 = |\cC_1| = 2^{\sqrt{l_1}}$. By taking $l_1$ large enough, we may assume that (\ref{stirling}) and (\ref{perm}) hold for all $n \ge N_1$.

For the induction step, suppose that a set $\cC_j$ of words of length $l_j$ has been constructed with cardinality $N_j = |\cC_j|$.
Give an ordering on $\cC_j$ and write $\cC_j = \{ u_i^{(j)} : 1 \leq i \leq N_j \}$.
Consider a new word $u_1^{(j+1)}$ formed by concatenating all the words in $\cC_j$ in order:
\[ u_1^{(j+1)} = u_1^{(j)} u_2^{(j)} u_3^{(j)} u_{4}^{(j)} u_5^{(j)} \cdots u_8^{(j)} u_{9}^{(j)} u_{10}^{(j)} \cdots u_{N_j}^{(j)}.  \]
Put $P_j = \{2\} \cup \{i^2: 2 \le i \le \sqrt{N_j}\}$ (We remark that the insertion of $2$ in $P_j$ is only for  notational convenience). The collection $\cC_{j+1}$ consists of those words of length $l_{j+1} = l_j \cdot N_j$ obtained by permuting the subwords $u_{i}^{(j)}$ of $u_1^{(j+1)}$ for $i \in P_j$ and leaving the others fixed.
That is, a typical element in $\cC_{j+1}$ is of the form
\[ u_1^{(j)} u_{\pi(2)}^{(j)} u_3^{(j)} u_{\pi(4)}^{(j)} u_5^{(j)} \cdots u_8^{(j)} u_{\pi(9)}^{(j)} u_{10}^{(j)} \cdots u_{N_j}^{(j)},  \]
where $\pi$ is a permutation on the set $P_j$.
We call a word $u_i^{(j)}$ in $\cC_j$ \emph{permuted} (or \emph{unstable}) if $i \in P_j$ and \emph{unpermuted} (or \emph{stable}) otherwise.

Note that the cardinality $N_j$ of $\cC_{j}$ and  the length $l_j$ of the words in $\cC_j$ satisfy the iterative formulas:
\begin{equation}\label{iter}
N_{j+1} = ( \sqrt{N_{j}})! \quad \mbox{and} \quad l_{j+1} = l_{j} \cdot N_{j} \quad \mbox{for } j \ge 1.
\end{equation}

Since $u_1^{(j)}$ is a prefix of $u_1^{(j+1)}$ for each $j$, there is a unique limit point $w \in \{ 0, 1 \}^\setN$ of the sequence $\{ u_1^{(j)} \}_{j \in \setN}$. Let $X^+$ be the orbit closure of $w$ and $X$ the inverse limit of $X^+$. Equivalently, we may let $X$ be the set of all bi-infinite sequences over $\{0, 1\}$ each word of which is a subword of a word in $\cC_j$ for some $j \in \setN$.

Since each word in $\cC_j$ occurs in every word in $\cC_{j+1}$, every word in $X$ occurs syndetically in $X$, so it follows that $X$ is minimal. Also any word $u \in \cC_j$, $j \in \setN$, occurs exactly once in $\cC_{j+1}$ (with respect to $\sigma^{l_j}$), hence an irreducible component of $\sigma^{l_j}(X)$ is uniquely ergodic. As in Lemma \cite[Lemma 1.9]{Gri72a}, we see that $X$ is uniquely ergodic. It follows that $X$ is strictly ergodic.

\begin{lem}\label{basic} Let $N_j$ and $l_j$ be as in the above. Then
\begin{equation*}
\lim_{j \to \infty} \frac{\log {N_{j}}}{\sqrt{l_j}} = 0 \quad \mbox{and} \quad \lim_{j \to \infty} \frac{\log {N_{j}}}{l_j^\beta} = \infty \mbox{ for any } \beta<1/2.
\end{equation*}
\end{lem}
\begin{proof} It suffices to prove the lemma only when each $N_j$ is a square number. We use the mathematical induction together with (\ref{stirling}) and (\ref{iter}) to see that
$$\log N_j \le \frac{\log 2}{2^{j-1}} \sqrt{l_j}.$$
The first equality of the lemma is immediate from this.
Once again we apply the mathematical induction and use (\ref{stirling}) and (\ref{iter}) to obtain
$$\log N_j \ge  \frac{\log 2}{2^{j-1}}\sqrt{l_j} - \frac{1}{2^{j-2}}\frac{\sqrt{l_j}}{\sqrt{l_1}} - \frac{1}{2^{j-3}}\frac{\sqrt{l_j}}{\sqrt{l_2}} - \cdots - \frac{\sqrt{l_j}}{\sqrt{l_{j-1}}}.$$
This yields the second equality of the lemma.
\end{proof}

We now show that the constructed system $X$ has topological entropy dimension $1/2$. In the remaining part of this section we assume that each $N_j$ is a square number to simplify the argument. Let $l \in \setN$. There is $j \in \setN$ such that $l_j \leq l < l_{j+1} = l_j \cdot N_j$.
Without loss of generality, we may assume that $l = k \cdot l_j$ for $1 \leq k < N_j$. Then by considering the number of the length $l$ prefixes of the words in $\cC_j$ and using (\ref{perm}), we have
    \begin{equation}\label{calculation_ED_geq}
        \begin{split}
            \frac{1}{\sqrt{l}} \log |\B_l(X)| & \geq \frac{1}{\sqrt{k l_j}} \log ( \Perm{\sqrt{N_j}}{\floor{\sqrt{k}}} ) \\
            &\ge \frac{1}{\sqrt{k l_j}} ( \floor{\sqrt{k}} \log \sqrt{N_j} - \floor{\sqrt{k}}) \\
            &\ge \frac{1}{4 \sqrt{l_j}} \log {N_j} - \frac{1}{\sqrt{l_j}}.
                    \end{split}
    \end{equation}
On the other hand, for each $u \in \B_l(X)$, $u$ is a subword of a word $B_1 \cdots B_{k+1}$ formed by concatenating $k+1$ words in $\cC_j$, and the maximal number of permuted positions in the $j$-th level is $\floor{\sqrt{k+1}}$. Hence we have
    \begin{equation}\label{calculation_ED_leq}
        \begin{split}
            \frac{1}{\sqrt{l}} \log |\B_l(X)| &\leq \frac{1}{\sqrt{k l_j}} \log ( l_{j+1} \cdot \Perm{\sqrt{N_j}}{\floor{\sqrt{k+1}}} ) \\
            &\leq \frac{1}{\sqrt{k l_j}} \log l_{j+1} + \frac{1}{\sqrt{k l_j}} \log \sqrt{N_j}^{\floor{\sqrt{k+1}}} \\
            &\leq \frac{1}{\sqrt{k l_j}} \log l_j + \frac{1}{2\sqrt{l_j}} \left( \frac{1}{\sqrt{k}} + \frac{\sqrt{k+1}}{\sqrt{k}} \right) \log N_j \\
            &\leq \frac{1}{\sqrt{k l_j}} \log l_j + \frac{3}{2 \sqrt{l_j}} \log N_j.
        \end{split}
    \end{equation}
From (\ref{calculation_ED_geq}), (\ref{calculation_ED_leq}) and the first equality of Lemma \ref{basic} it follows that
$$\lim_{l \to \infty} \frac{1}{\sqrt{l}} \log |B_l(X)| = 0.$$
A similar calculation together with the second equality of  Lemma \ref{basic} yields that
$$  \lim_{l \to \infty} \frac{1}{l^\beta} \log |B_l(X)| = \infty$$
for every $\beta < 1/2$.
This proves that the system $X$ has topological entropy dimension $1/2$.

In order to construct a system with topological entropy dimension $\alpha$ for given $\alpha \in (0,1)$, we let $\cC_{j+1}$ consist of the words obtained by permuting $\cC_j$-words of $u_1^{(j+1)}$ at positions $i^{\floor{1/\alpha}}$ for $1 < i \leq (N_j)^\alpha$. If $l_1$ and $N_1$ are large enough, then the proof goes as in the case $\alpha = 1/2$.

\begin{thm}
    Let $\alpha \in (0,1)$. There is a strictly ergodic subshift with entropy dimension $\alpha$.
\end{thm}

For each ordering on $\cC_j$, $j \in \setN$, we obtain a different strictly ergodic system. The subshifts obtained in this way are in general not topologically conjugate. However, they are all measure theoretically isomorphic.

\vspace{0.4cm}
\section{Size of atoms and return time}\label{sec:SMB}
%\section{A Range of Shannon-McMillan-Brieman property and the return time property}

In this section we consider a variant of the subshift constructed in Section \ref{sec:construction} as follows: Let $\cC_1$ be a set of binary words of length $l_1$. We assume that all members of $\cC_1$ start with $001$, and $00$ is not a subword of any element of $\cC_1$ but a prefix.
 %starting with $001$ and has no $00$ as a subword other than the first three symbols of it.
The rest of the construction is the same as the one in Section \ref{sec:construction}. The word $001$ serves as a marker, and the adjustment is adopted to control the return time of the words occuring in the constructed system $X$.

We will look into the size of atoms of the iterated partition and the first return property of the subshift $X$. It is well known that the Shannon-McMillan-Breiman theorem and the Ornstien-Weiss return time property hold for the systems of positive entropy. Also almost every irrational rotation map on a circle has the analogous properties \cite{KimP08}. Unfortunately, it is not the case for our system. However, it exhibits some `regularity' in the size of atoms and the first return time.
%We will look into the size of atoms of the iterated partition and the return time property of the new subshift $X$. We recall that a positive entropy system has Shannon-McMillan-Breiman theorem. Also almost every irrational rotation map on a circle has the similar property \cite{KimP08}. However, we show that our system of entropy dimension $\alpha$ does not have the corresponding property:
%$ \inf \{ \beta : \limsup - \frac{1}{n^\beta} \log {\mu(P_n(x))} \} = \inf \{ \beta : \liminf - \frac{1}{n^\beta} \log {\mu(P_n(x))} \} = \alpha$. But they exhibit some `regularity' in their size of atoms and return times.

Let $\mu$ be the unique $\sigma$-invariant measure on $X$. As usual, for $u \in \B(X)$, denote by $\mu(u)$ the measure of the cylinder $[u] = \{ x \in X : x_{[0,|u|)} = u \}$. Since $X$ is uniquely ergodic, $\mu(u) = \mu([u])$ is equal to the limit of the relative frequency of $u$ in $w_{[1,n]}$, where $w$ is the unique limit point obtained in the process of constructing $X$.

For each $x \in X$, denote by $P_n(x)$ the $n$-cylinder $[x_0 \cdots x_{n-1}] = \{ y \in X : y_i = x_i \text{ for } 0 \leq i < n \}$ of $x$. Also denote by $R_n(x)$ the first return time of $x$ to the $n$-cylinder containing $x$, i.e., $R_n(x) = \inf \{ k > 0 : x_{[k,k+n)} = x_{[0,n)} \}$. Since $X$ is minimal, we have $R_n(x) < \infty$ for each $x \in X$ and $n \in \setN$.

The first result states that the subexponential growth rate of the measure of $P_n(x)$ can be any number between $0$ and the topological entropy dimension 1/2 (Theorem \ref{thm:SMB}).
We begin with simple lemmas.

\begin{lem}\label{lem:occur_after_multiple_of_li}
    Let $x \in X$ and $j \in \setN$. If $x_{[0,l_j)}$ is a word in $\cC_j$, then each return time of $x$ to $P_{l_j}(x)$ is a multiple of $l_j$. %That is, if $x_{[m,m+l_j)} = x_{[0,l_j)}$, then $m$ is a multiple of $l_j$.
In particular, $R_{l_j}(x)$ is a multiple of $l_j$.
\end{lem}
\begin{proof}
    For $j = 1$, $x_{[0,l_1)}$ begins with the marker $001$. Since $00$ does not occur in a word of $\cC_1$ except at its prefix, the word $001$ cannot occur in the middle of $CC' \in \cC_1 \cC_1$. It follows that each return time of $x$ to $P_{l_1}(x)$ is a multiple of $l_1$.
%$R_{l_1}(x)$ is a multiple of $l_1$. By considering the first return times of $\sigma^{k \cdot l_1} (x)$, for $k \in \setN$,

    Suppose that the claim holds for $j-1$. If $x_{[0,l_j)}$ is in $\cC_j$, then it begins with $v = u^{(j-1)}_1 \in \cC_{j-1}$. Let $m$ be a return time of $x$ to $P_{l_j}(x)$. Then by the induction hypothesis, $m$ is a multiple of $l_{j-1}$. However, all the words in $\cC_{j-1}$ other than $v$ are different from $v$, hence the first coordinate of $v$ should occur in $x_{(1,\infty)}$ at positions which are multiples of $l_j$. It follows that $m$ also should be a multiple of $l_j$.
\end{proof}

\begin{lem}\label{lem:unique_decomp}
    For each $x \in X$ and $j \in \setN$, there is a unique decomposition of $x$ into $\cC_j$ words. %(and $b$'s).
\end{lem}
\begin{proof}
	The existence follows from compactness argument and the uniqueness follows from the existence of the marker word $001$ and the fact that $u^{(j-1)}_1$, which is the first member of $\cC_{j-1}$, is the prefix of each word of $\cC_j$ for $j \in \setN$. %added in each step.
\end{proof}

\begin{lem}\label{lem:same_measure}
    Let $C_1, \cdots, C_p$ be unpermuted words in $\cC_j$. Suppose that $C_1 \cdots C_p$ be a subword of some element of $C_{j+1}$. Then we have $\mu(C_1) = \mu(C_1 \cdots C_p)$.
\end{lem}
\begin{proof}
	By Lemma \ref{lem:occur_after_multiple_of_li}, in $w$ both $C_1$ and $C_1 \cdots C_p$ occur at the positions which are multiples of $l_j$. The result follows from the fact that the relative frequencies of $C_1$ and $C_1 \cdots C_p$ in $w$ are the same.
\end{proof}
%The above lemma follows from the observation that the relative frequencies of $C_1$ and $C_1 \cdots C_p$ in $w$ are the same. Because of Lemma \ref{lem:occur_after_multiple_of_li}, we only need to check the number of occurrences at the positions which are multiples of $l_j$.

We now present the results on the growth rate of the size of atoms.

\begin{thm}\label{thm:SMB}
	Let $X$ be the subshift with topological entropy dimension 1/2 constructed in this section. Then the following hold.
    \begin{enumerate}
        %\item For each $x \in X$,
        \item For $\mu$-a.e. $x \in X$,
            \[ \limsup_{n \to \infty} - \frac{1}{n^\beta} \log {\mu(P_n(x))} \] has the critical value $1/2$.

        \item There is a set $\hat X \subset X$ of full measure with the property  that for any $\tau \in [0,1/2]$, there is an increasing sequence $\{n_j\}_{j \in \setN}$ such that for each $x \in \hat X$,
            \[ \lim_{j \to \infty} - \frac{1}{(n_j)^\beta} \log {\mu(P_{n_j}(x))} \] has the critical value $\tau$.
    \end{enumerate}
    %It has a variant of the Shannon-McMillan-Brieman property
\end{thm}

%We begin with the following observations.

\begin{proof}%[The proof of Theorem \ref{thm:SMB}]
    Let $x \in X$. We first show that $\limsup_{n \to \infty} - \frac{1}{n^\beta} \log {\mu(P_n(x))} $ has the critical value not greater than $1/2$. By Lemma \ref{lem:unique_decomp} there is a unique decomposition of $x$ into a concatenation of words of $\cC_j$ for each $j \in \setN$. We may write $x = \cdots B^{(j)}_{-1} B^{(j)}_0 B^{(j)}_1 B^{(j)}_2 \cdots$, where $B^{(j)}_i \in \cC_j$ for each $i \in \setZ$ and $B^{(j)}_0$ is the unique word of $\cC_j$ containing the $0$-th coordinate of $x$. %Denote by $B^{(j)}_0$ the unique word occurring at the positions containing $0$-th coordinate.
    Note that $B^{(j)}_0$ is a subword of $B^{(j+1)}_0$ for each $j$. %is the unique element in $\cC_j$, occurring at the coordinate $0$.

    For each $n \in \setN$, there is a maximal $j$ such that $P_n(x)$ is a subword of $B^{(j)}_0 \cdots B^{(j)}_{k-1}$ with $k \geq 3$ and $B^{(j)}_1 \cdots B^{(j)}_{k-2}$ is a subword of $P_n(x)$. By the choice of $j$ we also have $k \leq 2 N_j$. Then since the maximal number of permuted words in $\cC_j$ occurring in $B^{(j)}_0 \cdots B^{(j)}_{k-1}$ is $\sqrt{k}$, we have
    \[
    \begin{split}
        \mu(P_n(x)) & \geq \mu(B^{(j)}_0 B^{(j)}_1 \cdots B^{(j)}_{k-1} ) \\
                    & = \lim_{l \to \infty} \frac{1}{l} \cdot \text{ number of occurrences of $B^{(j)}_0 B^{(j)}_1 \cdots B^{(j)}_{k-1}$ in $w_{[1,l]}$} \\
                    & \geq \frac{(\sqrt{N_j} - \sqrt{k})!}{l_{j+2}} = \frac{(\sqrt{N_j} - \sqrt{k})!}{l_{j+1} (\sqrt{N_j})!} \\
                    & \geq \frac{1}{l_{j+1}} ({\sqrt{N_j}}^{\sqrt{k}})^{-1}                    = \frac{1}{l_j} (N_j ({\sqrt{N_j}}^{\sqrt{k}}))^{-1} = (l_j ({\sqrt{N_j}}^{\sqrt{k} + 2}))^{-1}.
    \end{split}
    \]
    %As in the calculation of the topological dimension in Section \ref{section:Construction}, if $l_j \leq n < l_{j+1}$ we may assume that $n = k l_j$ for some $1 \leq k < N_j$. Hence
    Since $n \geq (k-2)l_j$, we obtain
    \[
        - \frac{1}{n^\beta} \log {\mu(P_n(x))} \leq \frac{1}{((k-2) l_j)^\beta} \log (l_j \cdot \sqrt{N_j}^{\sqrt{k} + 2}) \sim \frac{1}{2} \frac{\sqrt{k} + 2}{((k-2) l_j)^\beta } \log N_j.
    \]
    Since the limit of the last sequence has the same critical value as the limit of $\frac{1}{{(l_j})^\beta} \log {N_{j}}$, the limit superior in (1) has the critical value not greater than $1/2$. On the other hand, it is immediate from (2) with $\tau = \frac{1}{2}$ that the critical value of the limit superior is not less than $1/2$.

    \vspace{0.1cm}

    %For the proof of (2), we first claim that for $\mu$-a.e. $x \in X$, $B^{(j)}_0 (x)$ is an unpermuted word in $\cC_j$ for all large $j \in \setN$.

    Now we prove the second statement of the theorem. Let $\tilde X$ be the set of $x \in X$ such that $B_0^{(j)}(x)$ is an unpermuted word for all sufficiently large $j \in \setN$. We claim that $\mu(\tilde X) = 1$.  Let $L_j$ be the set of all $x \in X$ such that $B^{(j)}_0 (x)$ is a permuted word in $\cC_j$. Then we have $\mu(L_j) = \frac{\sqrt{N_j}}{N_j}$, so $\sum_j \mu(L_j) < \infty$. Now the Borel-Cantelli lemma proves the claim.

   % Note that each word $u$ in $\cC_{j+1}$ has $\sqrt{N_j} - 1$ permuted $\cC_j$-words, and in $u$ there are $\sqrt{N_j}$ (or $\sqrt{N_j} - 1$) sequences of consecutive unpermuted $\cC_j$-words. Note that the lengths of such sequences increase by $2 l_j$ except the first and the last ones, and the maximal length of such unpermuted words is about $2 \sqrt{N_j} \cdot l_j$.

    From now on we will assume that $N_j$ is a square number since, if not, a slightly modified argument will work. Note that every word in $\cC_{j+1}$ contains $\sqrt{N_j}$ isolated permuted $\cC_j$-words and there is a chain of consecutive unpermuted $\cC_j$-words between two adjacent permuted $\cC_j$-words. Let $\cU_j$ be the collection of those chains, that is,
    \[
        \cU_j = \{ u_1^{(j)}, u_3^{(j)}, u_5^{(j)} \cdots u_8^{(j)}, u_{10}^{(j)} \cdots u_{15}^{(j)}, \cdots, u_{i^2 + 1}^{(j)} \cdots u_{(i+1)^2 - 1}^{(j)}, \cdots \}
    \]
    Then the cardinality of $\cU_j$ is $\sqrt{N_j}$, and the numbers of unpermuted $\cC_j$-words in the chains in $\cU_j$ are $1, 1, 4, 6, 8, 10, \cdots, 2(\sqrt{N_j} - 1)$.
    %Note that every word in $\cC_{j+1}$ contains $\sqrt{N_j}$ isolated permuted $\cC_j$-words, and $\sqrt{N_j}$ sequences of consecutive unpermuted $\cC_j$-words. The numbers of unpermuted $\cC_j$-words in the sequences are $1, 1, 4, 6, 8, 10, \cdots, 2(\sqrt{N_j} - 1)$.

    Given $x \in \tilde X$ and $j \in \setN$ with $B^{(j)}_0 (x)$ unpermuted, let $C_1 \cdots C_{p(x,j)}$ be the chain in $\cU_j$ containing $B^{(j)}_0 (x)$. Also denote by $q(x,j)$ the integer with $C_{q(x,j)} = B^{(j)}_0 (x)$. Note that $p(x,j)$ and $q(x,j)$ are well defined for all large $j \in \setN$.

    %Given a fixed such $j$, let $C_1 \cdots C_{p(x,j)}$ be the maximal consecutive unpermuted words of $\cC_j$ lying in $B^{(j+1)}_0 (x)$.

    Fix $\frac{1}{4} < \eta < \frac{1}{2}$. For each $j \in \setN$, let $T_j$ be the set of all $x \in \tilde X$ such that $p(x,j) > (N_j)^\eta$ and $q(x,j) < p(x,j) - (N_j)^\frac{1}{4}$. Then we have %for \gamma = \min(1/4, 1-2\eta),
    \[
		\mu({T_j}^C) < \frac{1}{N_j} ( \sum_{i \leq {N_j^{\eta}} / 2} 2i + \sqrt{N_j} \cdot (N_j)^\frac{1}{4} )
		%\sum_{i = 1}^{ (N_j)^\eta } i + \sqrt{N_j} \cdot (N_j)^\frac{1}{4} < (N_j)^\gamma
    \]
    for all large $j$. Let $\hat X$ be the set of all $x \in \tilde X$ such that $x \in T_j$ for all large $j \in \setN$. Since $\sum_{j \in \setN} \mu({T_j}^C) < \infty$, by the Borel Cantelli Lemma we have $\mu(\hat X) = 1$.
    %and let $X'$ be the set of all point $x \in \tilde X$ such that $p(x,j) > (N_j)^\eta$ for all large $j \in \setN$. Since the set of all $x$ with $p(x,j) \leq (N_j)^\eta$ has measure $< \frac{(N_j)^{2\eta}}{N_j}$ and $\sum \frac{1}{{N_j}^\gamma} = 0$ for all $\gamma > 0$, we also have $\mu(X') = 1$ by Borel-Cantelli lemma.
    %Finally, let $\hat X$ be the set of all $x \in X'$ such that $q(x,j) < p(x,j) - \sqrt[4]{N_j}$ for all large $j \in \setN$. A similar application of Borel-Cantelli lemma gives that $\mu(\hat X) = 1$.
    %Intuitively, $\hat X$ consists of the points
    Intuitively, for any $x \in \hat X$, the word $B^{(j)}_0 (x)$ lies in the forepart of a long chain of consecutive unpermuted $\cC_j$-words in $B^{(j+1)}_0 (x)$.
    Note that for all $x \in \hat X$, $B^{(j)}_0 (x) \cdots B^{(j)}_{\sqrt[4]{N_j}}(x)$ is a chain of unpermuted $\cC_j$-words for all large $j \in \setN$.

%The following figure visualizes a typical location of $B^{(j)}_0(x)$ in $B^{(j+1)}_0(x)$ for $x \in \hat X$.
%Note that for $x \in \hat X$, $B^{(j)}_0 (x) \cdots B^{(j)}_{\sqrt[4]{N_j}}(x)$ is a sequence of unpermuted $\cC_j$-words for all large $j \in \setN$.

%\vspace{1cm}
%\begin{center}
%Figure 1 (to be added)
%\end{center}

%\vspace{1cm}

    Let $\tau \in [0,\frac{1}{2}]$ be given. % \alpha = 1/2 in this proof
    Since $\lim_{j \to \infty} \frac{1}{(l_j)^\beta} \log {N_j}$ has the critical value $\beta = 1/2$ and $\lim_{j \to \infty} \frac{1}{(l_j \cdot \sqrt[4]{N_j})^\beta} \log {N_j} = 0$ for all $\beta > 0$, there is a sequence ${n_j}$ such that $l_j \leq n_j \leq l_j \sqrt[4]{N_j}$ and $\lim_{j \to \infty} \frac{1}{(n_j)^\beta} \log {N_j}$ has the critical value $\tau$.
    For any $x \in \hat X$ and $j \in \setN$, by considering a $\cC_j$-word envelop of $P_{n_j}(x)$ as in the proof of (1), we see that $\mu(P_{n_j}(x))$ is almost same as $\mu(B^{(j)}_0 (x) \cdots B^{(j)}_p(x))$ for some $p$, which in turn is the same as $\mu(B^{j}_0 (x))$ by Lemma \ref{lem:same_measure}.
    %For any $x \in \hat X$ and $j \in \setN$, by considering a $\cC_j$-word envelop of $P_{n_j}(x)$ as in the proof of (1), $\mu(P_{n_j}(x))$ is almost same as $\mu(B^{(j)}_0 (x) \cdots B^{(k)}_p(x))$ for some $k$ and $p$, which in turn is same as $\mu(B^{k}_0 (x))$ by Lemma \ref{lem:same_measure}.
    Thus for each $x \in \hat X$, the sequence $\lim_{j \to \infty} - \frac{1}{(n_j)^\beta} \log {\mu(P_{n_j}(x))}$ has the same critical value as $\lim_{j \to \infty} \frac{1}{(n_j)^\beta} \log {N_j}$. This completes the proof of (2).
\end{proof}
%we have $p > 1/\sqrt{N_j}$, i.e., $B^{(j)}_0(x)$ lies in a sufficiently long piece of consecutive unpermuted words. By an application of Borel-Cantelli lemma we also have $\mu(M) = 1$.
%Finally, let $\hat X \subset M$ be the set of all $x \in M$ such that $B^{(j)}_0(x)$ lies in the first part of consecutive permuted words. If the part has a small measure compared to the length of consecutive unpermuted words, then we have $\mu(\hat X) = 0$. %A precise value will be given later.
%Now let $\tau \in [0,\alpha]$.

%Let $\cP$ be the $0$-th coordinate partition for $X$ and denote by $H(\cQ)$ the usual metric entropy for a partition $\cQ$ of $X$. Then Theorem \ref{thm:SMB} directly implies that $\limsup_{n \to \infty} \frac{1}{n^\beta} H(\bigvee_{i=0}^{n-1} \sigma^{-i} \cP)$ has the critical value $1/2$ and that $\liminf_{n \to \infty}$ $\frac{1}{n^\beta} H(\bigvee_{i=0}^{n-1} \sigma^{-i} \cP) = 0$ for all $\beta > 0$.

\begin{rem}
    (1) Let $\cP$ be the $0$-th coordinate partition for $X$ and denote by $H(\cQ)$ the usual metric entropy for a partition $\cQ$ of $X$. Then Theorem \ref{thm:SMB} directly implies that $\limsup_{n \to \infty} \frac{1}{n^\beta} H(\bigvee_{i=0}^{n-1} \sigma^{-i} \cP)$ has the critical value $1/2$ and that $\liminf_{n \to \infty}$ $\frac{1}{n^\beta} H(\bigvee_{i=0}^{n-1} \sigma^{-i} \cP) = 0$ for all $\beta > 0$. However, this does not guarantee that the system has metric entropy dimension $1/2$. Indeed, our system has metric entropy dimension $0$. (For the definition of metric entropy dimension, see \cite{DouHP14}).

    (2) It is not known whether Theorem 4.4 holds for general dynamical systems of entropy dimension $\alpha > 0$. We conjecture that the critical value of $\limsup$ equals topological entropy dimension and that of $\liminf$ equals metric entropy dimension. %whether $\limsup_{n \to \infty} - \frac{1}{n^\beta} \log {\mu(P_n(x))}$ attains topological entropy dimension as its critical value and $\liminf_{n \to \infty} - \frac{1}{n^\beta} \log {\mu(P_n(x))}$ does metric entropy dimension as its critical value.
\end{rem}

    %By looking at the partition of $\bigvee_{i=0}^{n-1} T^{-i} \cP$, it is not difficult to see that all of our systems have metric entropy dimension 0. It is not clear whether $\limsup_{n \to \infty} - \frac{1}{n^\beta} \log {\mu(P_n(x))}$ attains topological entropy dimension as its critical value and $\liminf_{n \to \infty} - \frac{1}{n^\beta} \log {\mu(P_n(x))}$ does metric entropy dimension as its critical value.
%\vspace{0.1cm}
%\newpage
We shall show similar results hold for the first return time. %Given a sequence $\{a_n\}_{n \in \setN}$, there is at most one
We begin with a lemma.
\begin{lem}\label{lem:SMB_to_returntime_upperbound}
    Let $X$ be a subshift and $\mu$ be an invariant measure on $X$. Suppose that $\limsup_{n \to \infty} - \frac{1}{n^\beta} \log {\mu(P_n(x))}$ has the critical value $c$ for $\mu$-a.e. $x \in X$. Then the critical value of $\limsup_{n \to \infty} \frac{1}{n^\beta} \log {R_n(x)}$ is less than or equal to $c$ for $\mu$-a.e. $x \in X$.

    %It has a variant of the Shannon-McMillan-Brieman property
\end{lem}
\begin{proof}

    %Denote by $\cP$ the $0$-th coordinate partition for $X$.
    Let $\epsilon > 0$ be given. Take $\alpha' > \alpha > \gamma > c$. Fix $n \in \setN$. For each $D \in \bigvee_{i=0}^{n-1} \sigma^{-i} \cP$, consider the set
    \[ D_0 = \{ x \in D : R_n(x) > \exp(n^{\alpha}) \}. \]
    Then for all $1 \leq i, j \leq \exp(n^{\alpha})$ with $i \neq j$, the sets $\sigma^i D_0$ and $\sigma^j D_0$ do not intersect. So we have $\mu(D_0) < \exp(-n^{\alpha})$.
    Now, let
    \[
            Q_n = \{ x \in X :~~  \mu(P_n(x)) > \exp(-n^\gamma)  \text{ and } R_n(x) > \exp(n^{\alpha}) \}.
            %Q_n = \{ x \in X :~~ & x \in D \text{ for some } D \in \bigvee_{i=0}^{n-1} \sigma^{-i} \cP \text{ with $\mu(D) > \exp(-n^\gamma)$} \\
            %&\text{and } R_n(x) > \exp(n^{\alpha}) \}.
    \]
    Since the number of $D$'s in $\bigvee_{i=0}^{n-1} \sigma^{-i} \cP$ with $\mu(D) > \exp(-n^\gamma)$ is at most $\exp(n^{\gamma})$, the measure of $Q_n$ can be estimated by
    \[ \mu(Q_n) \leq e^{-n^{\alpha}} \cdot e^{n^{\gamma}} = e^{n^{\gamma} - n^{\alpha}}, \]
    so we have $\sum \mu(Q_n) < \infty$. By the Borel-Cantelli lemma, it follows that for $\mu$-a.e. $x \in X$, $x \notin Q_n$ for all large $n$. %Hence $R_n(x) \leq \exp(n^\alpha)$ for all large $n$.
    %\begin{equation}\label{eq:SMB_to_returntime}
    %\end{equation}

    By the assumption, for $\mu$-a.e. $x \in X$ we have $\limsup_{n \to \infty} - \frac{1}{n^\gamma} \log {\mu(P_n(x))} = 0$. Hence $\mu(P_n(x)) > \exp(-n^{\gamma})$  for all large $n \in \setN$. As $x \notin Q_n$ for all large $n \in \setN$, we have $R_n(x) \leq \exp(n^{\alpha})$ for such $n \in \setN$.
    %By Theorem \ref{thm:SMB} (1), for large $n \in \setN$, $\mu(D) > \exp(-n^{\gamma})$ for elements $D$ in $\bigvee_{i=0}^{n-1} \sigma^{i} \cP$. For $\mu$-a.e. $x \in X$, for large $n$, we have $x \in D$ with $D$ satisfying the property \ref{eq:SMB_to_returntime}. Hence the number of such $D$'s in $\cP_n$ is at most $\exp(n^{\gamma})$.
    %Let $Q_n = \{ x \in X : x \in D \text{ satisfying (\ref{eq:SMB_to_returntime}) and } R_n(x) > \exp(n^{\alpha}) \}$. Then the measure of $Q_n$ can be estimated by
    It follows that
        \[ \limsup_{n \to \infty} \frac{\log R_n(x)}{n^{\alpha'}} = 0 \]
    for a.e. $x \in X$. Since the last inequality holds for all $\alpha' > c$, the proof is complete.
\end{proof}

\begin{thm}\label{thm:RT}
    Let $X$ be as in Theorem \ref{thm:SMB}. Then the following hold.
    \begin{enumerate}
        \item For $\mu$-a.e. $x \in X$,
            \[ \limsup_{n \to \infty} \frac{1}{n^\beta} \log {R_n(x)} \] has the critical value $1/2$.
        \item There is a set $\hat X \subset X$ of full measure with the property  that for any $\tau \in [0,1/2]$, there is an increasing sequence $\{n_j\}_{j \in \setN}$ such that for each $x \in \hat X$,
            \[ \lim_{j \to \infty} \frac{1}{(n_j)^\beta} \log {R_{n_j}(x)} \] has the critical value $\tau$.
    \end{enumerate}
    %It has a variant of the Shannon-McMillan-Brieman property
\end{thm}

\begin{proof}

    By Lemma \ref{lem:SMB_to_returntime_upperbound} and Theorem \ref{thm:SMB}(1) it follows that the limsup in (1) has the critical value no greater than $1/2$. The inequality of the other direction is an immediate consequence of the second result of the theorem with $\tau = 1/2$.
    %We will first show that the critical value for the return time is not greater than the critical value for the SMB-like limit supremum. %The proof resembles the one in \cite{OrnW93}.

    For the proof of (2), let $\hat X$ be the same set as in the proof of Theorem \ref{thm:SMB}. Given $\tau \in [0,1/2]$, take $\{n_j\}_{j \in \setN}$ as in (2) of Theorem \ref{thm:SMB}. Let $x \in \hat X$. There is a chain of unpermuted $\cC_j$-words $B^{(j)}_0 \cdots B^{(j)}_{k-1}$, $k \geq 3$, such that $P_{n_j}(x)$ is a subword of $B^{(j)}_0 \cdots B^{(j)}_{k-1}$ and $P_{n_j}(x)$ contains $B^{(j)}_1 \cdots B^{(j)}_{k-2}$ as a subword.
    %Then since the maximal number of permuted words occurring in $B^{(j)}_0 \cdots B^{(j)}_k$ is $\sqrt{k}$, we have
    Then both $B^{(j)}_0 \cdots B^{(j)}_{k-1}$ and $B^{(j)}_1 \cdots B^{(j)}_{k-2}$ occur exactly after the time $l_{j+1}$, so $R_n(x) \leq l_{j+1}$. By the uniqueness of the decomposition into $\cC_j$-words and by Lemma \ref{lem:occur_after_multiple_of_li} it follows that $R_{n_j}(x) = l_{j+1} = l_j N_j$. It is easy to see that the limit of $\frac{1}{{n_j}^\beta} \log R_{n_j}(x)$ has the critical value $\tau$.
    %The remaining part of the proof goes similarly as in Theorem \ref{thm:SMB}.
\end{proof}

%We remark that all the results also hold for the one-sided subshift $X^+$.

%\begin{rem}
    %By looking at the partition of $\bigvee_{i=0}^{n-1} T^{-i} \cP$, it is not difficult to see that all of our systems have metric entropy dimension 0. It is not clear whether $\limsup_{n \to \infty} - \frac{1}{n^\beta} \log {\mu(P_n(x))}$ attains topological entropy dimension as its critical value and $\liminf_{n \to \infty} - \frac{1}{n^\beta} \log {\mu(P_n(x))}$ does metric entropy dimension as its critical value.
%\end{rem}

%\vspace{0.3cm}
%\section{Entropy generating sequences and metric entropy dimension}

\vspace{0.4cm}
%\section{Rigidity and weakly mixing property}
\section{Further properties and remarks}\label{sec:further}

Sections 3 and 4 give a general method to construct strictly ergodic subshifts with positive entropy dimension. With slight modifications, we can construct such systems with specific topological and measure theoretic properties.

\subsection{Entropy generating sequence} %Let $(X,\sigma)$ be a subshift and $\cP$ be the $0$-th coordinate partition. An in
Let $S = \{s_i\}_{i \in \setN}$ be an increasing sequence of natural numbers. It is called an \emph{entropy generating sequence} for a subshift $(X,\sigma)$ if $\liminf_{n \to \infty} \frac{1}{n} \log \cN( \bigvee_{i=1}^{n} \sigma^{-s_i} \cP ) > 0$. The \emph{upper dimension} of $S$ is defined by
\[ \overline D(S) = \inf \{ \tau \geq 0 : \limsup_{n \to \infty} \frac{n}{(s_n)^\tau} = 0\}. \]
The intuitive idea of an entropy generating sequence is to specify positions where the independence occur.
It is known that the upper entropy dimension $\overline D(X,\sigma)$ defined in Section \ref{sec:background} equals the supremum of $\overline D(S)$ over all entropy generating sequences $S$ \cite[Theorem 3.10]{DouHP13}. For a subshift, there is an entropy generating sequence attaining the supremum. %In general such a sequence is not easy to describe.

It is intricate to describe an entropy generating sequence for a general system constructed in Section 3. However, imposing some additional conditions in the process of the construction, we can obtain a system of which an entropy generating sequence is rather obvious. Suppose that the $j$-th word set $\cC_j$ is constructed by the induction. Recall that $\cC_j$ consists of $N_j$ words of length $l_j$ and that $N_{j+1} = (\sqrt{N_j})!$ and $l_{j+1} = l_j N_j$. Any word $u$ in $\cC_j$ can be written uniquely as $u = u_1 u_2 \cdots u_{N_{j-1}}$ where $u_k$'s are words in $\cC_{j-1}$, and distinguished by the sequence $\{u_k\}_{k \in P_{j-1}}$, where  $P_{j-1} = \{2\} \cup \{i^2: 2 \le i \le \sqrt{N_{j-1}}\}$. We consider the subset
\[ \cD_j = \{ u \in \cC_j : \{ u_k \}_{k \in Q_{j-1}} \text{ are distinct } \} \]
where $Q_{j-1} = \{2\} \cup \{i^2: 2 \le i \le \sqrt{N_{j-1}}/2\}$. The cardinality of $\cD_j$ is $\Perm{\sqrt{N_{j-1}}}{\sqrt{N_{j-1}}/2}$. Now we order the elements of $\cC_j$ so that the permuted positions $k \in P_j$ are occupied by elements in $\cD_j$: if
\[ u_1^{(j+1)} = u_1^{(j)}u_2^{(j)}u_3^{(j)}u_4^{(j)}u_5^{(j)} \cdots u_{N_j}^{(j)} \]
is the word formed by concatenating all the words of $\cC_j$ in order, then $u_k^{(j)}$ are in $\cD_j$ for all $k \in P_j$. Put $S_1 = [0, l_1]$,
\[
    S_j = \bigcup_{k \in Q_j} \left( k \cdot l_{j-1} + S_{j-1} \right) ~~~~~~~~~\text{ (for $j > 1$})
\]
and $S = \bigcup_{j \in \setN} S_j$. Note that
\[ |S_j| = \frac{l_1}{2^{j-1}} \prod_{i=1}^{j-1} \sqrt{N_i}.  \]
By a straightforward calculation it can be seen that $\liminf_{j \to \infty} \frac{1}{|S_j|} \log N_j > 0$, and we have $\overline D(S) = 1/2$. %Since $|S_j| = |S_1| \cdot \frac{1}{2^{j-1}} \prod_{i=1}^{j-1} \left( \sqrt{N_i} - 1 \right)$, we have $\overline D(G) = \frac{1}{2}$. Also a calculation shows $\liminf_{j \to \infty} \frac{1}{|S_j|} \log N_j > 0$, which implies that $S = \bigcup_{j \in \setN} S_j$ is an entropy generating sequence.

%For arbitrary collections $\cC_j$, with a little effort, for each $j \in \setN$ we may also find half of permuted words in which the independence occur mostly in those positions. Summarizing this, we have the following result.
%\begin{prop}\label{prop:entropy_generating_sequence}
%    The subshift constructed in Section 2 has an entropy generating sequence $S$ with $\overline D(S) = \frac{1}{2}$.
%\end{prop}

%To construct an entropy generating sequence for a system with topological entropy dimension $\alpha$ for given $\alpha \in (0,1)$, let $S_1$ as before and start with $S_j = \bigcup_{k=2}^{{(N_{j-1})}^\alpha} \left( k^{\frac{1}{\alpha}} \cdot l_{j-1} + S_{j-1} \right)$. The construction of a sequence goes as in the case $\alpha = 1/2$.

\vspace{0.3cm}
\subsection{Weak mixing} The subshift constructed in Section 3 is not topologically weak mixing (hence not metrically weak mixing), since for any $u \in \cC_j$, the set of $k$'s with $\sigma^{kl_j} [u] \cap [u] = \emptyset$ has positive density.

By inserting a spacer symbol at each step, it is possible to construct a metrically and topologically mixing subshift. Let $b$ be a symbol other than $0$ and $1$. Given $\cC_j$ with $j \in \setN$, let $u_1^{(j+1)}$ be the concatenation of $\cC_j$ words with the spacers in front of each permuted words:
%In the remaining of this section, we consider a variant of the system constructed in Sections 2 and 3 and show that it is (metrically and topologically) weakly mixing and metrically rigid. Let $b$ be a symbol other than $0$ and $1$. This will be used as a spacer symbol. Given $j \in \setN$ and $\cC_{j}$, we let $u_1^{(j+1)}$ with the markers in front of each permuted words, so
    \[
        u_1^{(j+1)} = u_1^{(j)} u_2^{(j)} u_3^{(j)} b u_{4}^{(j)} u_5^{(j)} \cdots u_8^{(j)} b u_{9}^{(j)} u_{10}^{(j)} \cdots u_{N_j}^{(j)},
    \]
and the collection $\cC_{j+1}$ is obtained by permuting the subwords $u_{i^2}^{(j)}$ for $i > 1$ and leaving the others fixed. Hence a typical word in $\cC_{j+1}$ is of the form
\[ u_1^{(j)} u_2^{(j)} u_3^{(j)} b u_{\pi(4)}^{(j)} u_5^{(j)} \cdots u_8^{(j)} b u_{\pi(9)}^{(j)} u_{10}^{(j)} \cdots u_{N_j}^{(j)},  \]
where $\pi$ is a permutation on the set $\{ i^2 : 1 < i \leq \sqrt{N_j} \}$.
Since $b$ occurs rarely, it does not affect much on the calculations in the previous sections: Our new system $X$ is strictly ergodic, has entropy dimension $1/2$, and satisfies Theorem \ref{thm:SMB} and Theorem \ref{thm:RT}. A similar argument in \S 5.2 applies to $X$, hence $X$ is also metrically rigid.

As $X$ is strictly ergodic, metric weak mixing implies topological weak mixing. The existence of a spacer symbol forces $X$ to be weakly mixing by an approximation argument. Suppose that $f : X \to \setC$ is a measurable eigenfunction with eigenvalue $\lambda$. We may assume that $|f| = 1$. Let $\epsilon > 0$ be small enough. By letting $\cA_j = \{ \sigma^i([u]) : u \in \cC_j \text{ is unpermuted and } 0 \leq i < l_j \}$, it is easy to see that the Borel $\sigma$-algebra of $X$ is generated by $\cA_j$'s, $j \in \setN$. Note that $\cA_j$ does not cover $X$, but $\mu(\bigcup \cA_j ) > (1 - \frac{1}{l_j}) ( 1 - \frac{1}{\sqrt{N_j}})$.

There is a sequence of simple functions $f_j$ converging to $f$ and for each $j \in \setN$ and $f_j$ is constant on each element of $\cA_j$. Then there is $j$ with $\mu\{x: |f(x) - f_j(x)| < \epsilon \} > 1 - \epsilon$ and $\mu(\bigcup \cA_j) > 1 - \epsilon$. Let $F = \{x: |f(x) - f_j(x)| < \epsilon \}$.
 %$F \subset X$ with $\mu(F) > 1 - \epsilon$ and a simple function $f_j$ which is constant on each element in $\cA_j$ for some $j \in \setN$, such that $ |f(x) - f_j(x)| < \epsilon$ for $x \in F$. We may assume that $j$ is large so that $\mu(\bigcup \cA_j) > 1 - \epsilon$.

Now there is $B = \sigma^i([u]) \in \cA_j$ %, where $u \in \cC_j$ is an unpermuted word and $0 \leq i < l_j$,
such that $\mu(B \cap F) > (1- 2 \epsilon) \mu(B)$.
If not, then we have
    \[ 1 - 2\epsilon < \mu(F) - \epsilon < \mu(\bigcup \cA_j \cap F) < (1-2\epsilon) \mu(\bigcup \cA_j) < 1-2\epsilon, \]
which is a contradiction. Without loss of generality we may assume $i = 0$.

For each $w \in \cC_{j+2}$, we define
\[ \cI_w = \{ 0 \leq k < l_{j+2} : \sigma^{k}([w]) \subset B \}. \]
Then it is easy to see that $|\cI_w| = N_{j+1}$. An integer $k \in \cI_w$ is called \emph{good} if
$\mu(\sigma^k([w]) \cap F) > (1-\sqrt{2\epsilon}) \mu([w])$ and \emph{bad} otherwise. We will show that there exists some $v \in \cC_{j+2}$ such that $\cI_v$ has at least $(1-\sqrt{2\epsilon}) N_{j+1}$ good members.
\begin{proof}
Suppose not. Then for all $w \in \cC_{j+2}$ there are at least $\sqrt{2\epsilon} \cdot N_{j+1}$ bad members of $\cI_w$. Since $\sigma^k([w])$ is disjoint for $0 \leq  k < l_{j+2}$, we have
    \[
    \begin{split}
        \mu(\bigcup_{k \in \cI_w} \sigma^k([w]) \cap F) & \leq
        \mu(\bigcup_{\substack{ k \in \cI_w \\ k : good}} \sigma^k([w])) + \mu(\bigcup_{\substack{ k \in \cI_w \\ k : bad}} \sigma^k([w]) \cap F) \\
        & \leq (1-\sqrt{2\epsilon}) \mu(\bigcup_{k \in \cI_w} \sigma^k([w])) + \sqrt{2\epsilon} N_{j+1} \cdot ( 1 - \sqrt{2\epsilon} ) \mu([w]) \\
        & < (1-\sqrt{2\epsilon}) \mu(\bigcup_{k \in \cI_w} \sigma^k([w])) + \sqrt{2\epsilon} ( 1 - \sqrt{2\epsilon} ) \mu(\bigcup_{k \in \cI_w} \sigma^k([w])) \\
        & = ( 1 - 2\epsilon ) \mu(\bigcup_{k \in \cI_w} \sigma^k([w])).
    \end{split}
    \]
As $B$ is the disjoint union of all $\sigma^k([w])$'s with $w \in \cC_{j+2}$ and $k \in \cI_w$, by summing over all $w \in \cC_{j+2}$ we have $\mu(B \cap F) \leq (1-2\epsilon) \mu(B),$ which is a contradiction.
\end{proof}

Let $\cI_v = \{ k_1 < k_2 < \cdots < k_{N_{j+1}} \}$ be ordered. Note that $k_i - k_{i-1}$ is $l_{j+1} + 1$ if $i$ is a square number bigger than $1$, and $l_{j+1}$ otherwise.
%Note that in $D$, there are at least $(1-\sqrt{2\epsilon}) \cdot N_{j+1}$ \emph{good} indices $k$, i.e.,  we have $\mu(\sigma^k(D) \cap F) > (1-\sqrt{2\epsilon}) \mu(D)$ for such $k$. %if $u$ occurs in $D$, i.e., $\sigma^k(D) \subset B$ then $\mu(\sigma^k(D) \cap F) > (1-\sqrt{2\epsilon}) \mu(D)$ hence with probability $(1-\sqrt{2 \epsilon})$ we have $|f - f_j| < \epsilon$ on $\sigma^k(D)$.
    \vspace{0.2cm}
    \begin{claim}
        There are good members $k_{i_1}, k_{i_2}, k_{i_3}, k_{i_4} \in \cI_v$ and a positive integer $c$ such that $c \cdot l_{j+1} = k_{i_2} - k_{i_1} - 1 = k_{i_4} - k_{i_3}$.
    \end{claim}

    \begin{proof}
        Recall that there occur $(j+2)$-step spacers between $l^2$-th and $(l^2 + 1)$-st $\cC_{j+1}$-subblocks of $v$ for each $1 < l \leq \sqrt{N_{j+1}}$.

        Fix $\alpha = 2\sqrt{2\epsilon}$ and consider the first $\alpha N_{j+1}$ members of $\cI_v$. (For simplicity, we may assume that $\alpha N_{j+1}$ is an integer.) There we can find good members $k_{i_1} < k_{i_2}$ such that there is exactly one spacer between $k_{i_1}$-th and $k_{i_2}$-th $\cC_{j+1}$-subblocks of $v$. Otherwise the number of bad members of $\cI_v$ are greater than $\alpha /2 > \sqrt{2\epsilon}$, a contradiction. Hence we have $k_{i_2} - k_{i_1} = c \cdot l_{j+1} + 1$ for some integer $c$.
        As there are at most $\sqrt{\alpha N_{j+1}}$ spacers in the first $\alpha N_{j+1}$ members of $\cI_v$, we have $c < 4 \sqrt{\alpha N_{j+1}}$.

        %Now consider the last $(1-\alpha)N_{j+1}$ members of $\cI_v$.
        Suppose that there does not exist good members $k_{i_3}$ and $k_{i_4}$ such that $c \cdot l_{j+1} = k_{i_4} - k_{i_3}$.
        %then for each pair of members $(k, k+c)$ of $\cI_v$ between two square numbers $l^2$ and $(l+1)^2$, there is at least one bad member.
        Then at least one of $k_i$ and $k_{i+c}$ is bad if $(l-1)^2 \leq i, i+c < l^2$ for any $l \geq 2$.
        Therefore, the number of bad members of $\cI_v$ is bigger than
        %\frac{1}{N_{j+1}} \cdot \text{ all indices in $ 1 - \alpha$-portion } - \text{  elements which cannot be paired  } \\
        \[
            \begin{split}
                        & \sum_{l \leq \sqrt{N_{j+1}}} \Bigl\lfloor \frac{l^2 - (l-1)^2}{2c} \Bigr\rfloor \cdot c \\
                        & \geq \sum_{l \leq \sqrt{N_{j+1}}} \left( l - \frac{1}{2} - c \right) \\
                        & \geq \frac{N_{j+1}}{2} - c \sqrt{N_{j+1}} \\
                        & \geq \frac{1 - 8\sqrt{\alpha}}{2}  N_{j+1} > \sqrt{2\epsilon} N_{j+1}                       %& \geq \frac{1}{2} \cdot (  N_{j+1}(1-\alpha) - 2c \cdot ( \sqrt{ N_{j+1} } - ( \sqrt{\alpha N_{j+1}} ) ) \\
                        %& = \frac{1}{2} {N_{j+1}} \cdot \big( (1-\alpha) - (1-\sqrt{\alpha}) \cdot 8\sqrt{\alpha} \big) \\
                        %& = \frac{1}{2} {N_{j+1}} \big( 1 - 8\sqrt{\alpha} + 7\alpha \big) > \frac{1 - 8\sqrt{\alpha}}{2} > \sqrt{2\epsilon}
            \end{split}
        \]
        for all small $\epsilon > 0$, which is again a contradiction. Hence there are good members $k_{i_3}, k_{i_4} \in \cI_v$ such that $c \cdot l_{j+1} = k_{i_4} - k_{i_3}$.
    \end{proof}
%Therefore we may find four good indices $k_1 < k_2$ and $k_3 < k_4$ with $c \cdot l_{j+1} = k_2 - k_1 = k_4 - k_3 - 1$ for some $c \in \setN$. %To see this, there are two good approximable $k_1$ and $k_3$ in which there are only one spacer symbol $b$.

As $k_{i_1}$ and $k_{i_2}$ are good members of $\cI_v$, the set \[(\sigma^{k_{i_1}}([v]) \cap F) \cap \sigma^{-(k_{i_2} - k_{i_1})}(\sigma^{k_{i_2}}([v]) \cap F)\] is nonempty.
Take a point $y$ in the set above. Since $f$ is an eigenfunction with eigenvalue $\lambda$, we have $|f(y) - f \circ \sigma^{k_{i_2} - k_{i_1}}(y)| = |f(y) - f \circ \sigma^{c \cdot l_{j+1} + 1}(y)| = |f(y)| \cdot | 1 - \lambda^{c \cdot l_{j+1} + 1} | = | 1 - \lambda^{c \cdot l_{j+1} + 1} |$. Also, since $y \in (B \cap F) \cap \sigma^{-(k_{i_2} - k_{i_1})} (B \cap F)$ and $f_j$ is constant on $B$ we have
\[ |f(y) - f \circ \sigma^{k_{i_2} - k_{i_1}}(y)| \leq |f_j(y) - f_j \circ \sigma^{k_{i_2} - k_{i_1}}(y)| + 2\epsilon = 2 \epsilon. \]
Similarly by using $k_{i_3}$ and $k_{i_4}$ we have \[ | 1 - \lambda^{c \cdot l_{j+1}} | < 2 \epsilon, \]
and it follows that $| 1 - \lambda| = | \lambda^{c \cdot l_{j+1}} ( 1 - \lambda) | < 4 \epsilon$. This is true for every $\epsilon > 0$, hence we have $\lambda = 1$. Therefore $(X,\sigma)$ is metrically weakly mixing.

\vspace{0.2cm}
\subsection{Metric rigidity}
Let $X$ be any subshift constructed in this paper. We show that $X$ is metrically rigid with a rigidity sequence $\{l_n\}_{n \in \setN}$., i.e., $\lim_{n \to \infty} \mu(\sigma^{l_n} A \triangle A) = 0$ for each measurable set $A$.  To see this, first note that this holds for %$A \in \cA_j$, that is,
each set $A$ of the form $A = \sigma^k([u])$ with an unpermuted word $u \in \cC_j$ and $k \in \setN$.
For a measurable set $A$ and $\epsilon > 0$, there is $j \in \setN$ and finitely many $I_i \in \cA_j$ such that $\mu(A \triangle \bigcup I_i) < \epsilon$. Since each $I_i$ satisfies $\mu(\sigma^{l_n} I_i \triangle I_i) \to 0$, by an approximation we also have $\mu(\sigma^{l_n} A \triangle A) \to 0$ as $n \to \infty$.

 %let $C_0^j$ be the union of the base set of all the unpermuted words, i.e., $C_0^{(j)} = \{ \bigcup B_i^{(j)} : i \neq k^2 \text{ for } k = 2, \cdots, \sqrt{N_j} \}$. Then $\mu(C_0^j) > (1 - \frac{1}{\sqrt{N_j}}) \mu(B_0^j)$ and we have
%\[
    %\sigma^{l_n} (T^i C_0^j) \subset T^i C_0^j \text{ for } i = 0, 1, \cdots, l_n - 1,
%\]
%which shows that $X$ is rigid with a rigidity sequence $\{l_n\}_{n \in \setN}$.

\vspace{0.1cm}
\begin{ack*}
    The first named author was supported by 2012R1A1A2006874. The third named author was supported by NRF 2010-0020946.
\end{ack*}

%A similar argument in \S 5.2 applies to the system with a spacer, hence $X$ can be made to be both weakly mixing and rigid.
%If $f$ is constant on each level set of $\cC_j$, then consider the cylinder $A = [u_1^{(j)}]$.  Then there is a point $x \in A$ such that $x$ is of the form
%\[
    %x = ... u_1^{(j)} \cdots u_{N_j}^{(j)} u_1^{(j)} \cdots u_{N_j}^{(j)} b u_1^{(j)} \cdots u_{N_j}^{(j)} \cdots,
%\]
%so that $\sigma^{l_{j+1}}(x)$ and $\sigma^{2l_{j+1}+1}(x)$ are in $A$. Hence $f(x) = f(\sigma^{l_{j+1}}(x)) = f(\sigma^{2 l_{j+1} + 1}(x))$ and we have $\lambda = 1$.

%In general, given $\epsilon > 0$, we may use an approximation of an eigenfunction to a step function on a set of measure greater than $1 - \epsilon$, and a similar argument shows that $X$ is metrically weakly mixing.

%\begin{ack*}
%    The first named author was supported by 2012R1A1A2006874. The third named author was supported by NRF 2010-0020946.
%\end{ack*}

%\newpage
\bibliographystyle{amsplain}
\bibliography{_Bib_EntDim}

\providecommand{\bysame}{\leavevmode\hbox to3em{\hrulefill}\thinspace}
\providecommand{\MR}{\relax\ifhmode\unskip\space\fi MR }
\providecommand{\MRhref}[2]{%
  \href{http://www.ams.org/mathscinet-getitem?mr=#1}{#2}
}
\providecommand{\href}[2]{#2}
\begin{thebibliography}{10}

\bibitem{Bow10a}
L.~Bowen, \emph{\normalfont{A measure-conjugacy invariant for free group
  actions}}, {\it Ann. Math. (2)} \textbf{171} (2010), 1387--1400.

\bibitem{Bow10b}
\bysame, \emph{\normalfont{Measure conjugacy invariants for actions of
  countable sofic groups}}, {\it J. Amer. Math. Soc.} \textbf{23} (2010),
  217--245.

\bibitem{BoyL97}
M.~Boyle and D.~Lind, \emph{\normalfont{Expansive subdynamics}}, {\it Trans.
  Amer. Math. Soc.} \textbf{349} (1997), 55--102.

\bibitem{Cas03}
J.~Cassaigne, \emph{\normalfont{Constructing infinite words of intermediate
  complexity, Developments in language theory}}, {\it Lecture Notes in Comput.
  Sci.} \textbf{2450} (2003), 173--184.

\bibitem{Car97}
M.~de~Carvalho, \emph{\normalfont{Entropy dimension of dynamical systems}},
  {\it Portugal. Math.} \textbf{54} (1997), 19--40.

\bibitem{DouHP14}
D.~Dou, W.~Huang, and K.~K. Park, \emph{\normalfont{Entropy dimension of
  measure preserving systems}}, arXiv:1312.7225, 33pp.

\bibitem{DouHP13}
\bysame, \emph{\normalfont{Entropy dimension of topological dynamical
  systems}}, {\it Trans. Amer. Math. Soc.} \textbf{363} (2011), 659--680.

\bibitem{Fer97}
S.~Ferenczi, \emph{\normalfont{Measure-theoretic complexity of ergodic
  systems}}, {\it Israel J. Math.} \textbf{100} (1997), 189--207.

\bibitem{Goo74}
T.~Goodman, \emph{\normalfont{Topological sequence entropy}}, {\it Proc. London
  Math. Soc.} \textbf{29} (1974), 331--350.

\bibitem{Gri72a}
C.~Grillenberger, \emph{\normalfont{Constructions of strictly ergodic systems.
  I. Given entropy}}, {\it Z. Wahrscheinlichkeitstheorie und Verw. Gebiete.}
  \textbf{25} (1972/73), 323--334.

\bibitem{JohS15}
A.~Johnaon and A.~Sahin, \emph{\normalfont{Directional recurrence for infinite
  measure preserving $\mathbb{Z}^d$ actions}}, {\it Ergod. Th. Dynam. Sys.}
  \textbf{35} (2015), 2138--2150.

\bibitem{KamZ02}
T.~Kamae and L.~Zamboni, \emph{\normalfont{Maximal pattern complexity for
  discrete systems}}, {\it Ergod. Th. Dynam. Sys.} \textbf{22} (2002),
  1201--1214.

\bibitem{KimP08}
D.~Kim and K.~Park, \emph{\normalfont{The first return time properties of an
  irrational rotation}}, {\it Proc. Amer. Math. Soc.} \textbf{136} (2008),
  3941--3951.

\bibitem{KraHM14}
B.~Kra, B.~Host, and A.~Maass, \emph{\normalfont{Complexity of nilsystems and
  systems lacking nilfactors}}, {\it J. Anal. Math. } \textbf{124} (2014),
  261--295.

\bibitem{MelP05}
X.~Mela and K.~Petersen, \emph{\normalfont{Dynamical properties of the Pascal
  adic transformation}}, {\it Ergod. Th. Dynam. Sys.} \textbf{25} (2005),
  227--256.

\bibitem{Orn70}
D.~Ornstein, \emph{\normalfont{Bernoulli shifts with the same entropy are
  isomorphic }}, {\it Adv. Math.} \textbf{4} (1970), 337--352.

\bibitem{OrnW83}
D.~Ornstein and B.~Weiss, \emph{\normalfont{The Shannon-McMillan-Breiman
  theorem for a class of amenable groups}}, {\it Israel J. Math.} \textbf{44}
  (1983), 53--60.

\bibitem{OrnW93}
\bysame, \emph{\normalfont{Entropy and data compression schemes}}, {\it IEEE
  Trans. Inform. Theory} \textbf{39} (1993), 78--83.

\bibitem{RobS11}
E.~Robinson and A.~Sahin, \emph{\normalfont{Rank-one $\setZ^d$ actions and
  directional entropy}}, {\it Ergod. Th. Dynam. Sys.} \textbf{31} (2011),
  285--299.

\bibitem{Sha48}
C.~Shannon, \emph{\normalfont{A mathematical theory of communication}}, {\it
  Bell System Tech. J.} \textbf{27} (1948), 379--423, 623--656.

\bibitem{Shi}
P.~Shields, \emph{\textit{The ergodic theory of discrete sample paths}},
  Graduate Studies in Mathematics, vol.~13, American Mathematical Society,
  Providence, RI, 1996.

\end{thebibliography}

\end{document}